\documentclass[a4paper]{amsart}

\usepackage{amsmath,amsthm,amsfonts,amssymb,tikz,tikz-cd,enumerate,setspace,enumitem}
\usepackage{amsrefs}
\usepackage[T1]{fontenc}
\usepackage{parskip}
\usepackage{hyperref} 

\usetikzlibrary{arrows} 

\theoremstyle{plain}
\newtheorem{theorem}{Theorem}[section]
\newtheorem{lemma}[theorem]{Lemma}

\newtheorem{corollary}[theorem]{Corollary}

\theoremstyle{definition}

\newtheorem{definition}[theorem]{Definition}

\theoremstyle{remark}
\newtheorem{remark}[theorem]{Remark}

\newcommand{\C}{\mathbb{C}}             
\newcommand{\Z}{\mathbb{Z}}             
\newcommand{\Q}{\mathbb{Q}}                 
\newcommand{\cg}[1]{\langle #1 \rangle} 
\renewcommand{\H}{\mathcal{H}}    
\DeclareMathOperator{\End}{End}

\title{A Fixed Point Decomposition of Twisted Equivariant K-Theory}
\author{Tom Dove}\thanks{Tom Dove was supported by a PhD scholarship from the DAAD}
\author{Thomas Schick}
\author{Mario Vel\'{a}squez}\thanks{Part of this work was carried out during a visit of Mario Vel\'asquez in G\"ottingen supported by the DFG RTG "Fourier analysis and spectral theory"}

\address{Mathematisches Institut, Bunsenstrasse 3-5, 37073 G\"{o}ttingen, Germany}
\email{\href{mailto:thomas.dove@mathematik.uni-goettingen.de}{thomas.dove@mathematik.uni-goettingen.de}}
\email{\href{mailto:thomas.schick@math.uni-goettingen.de}{thomas.schick@math.uni-goettingen.de}}
\address{Departamento de Matem\'aticas\\Universidad Nacional de Colombia, sede Bogot\'a\\Cra. 30 cll 45 - Ciudad Universitaria\\ Bogot\'a, Colombia}
\email{\href{mailto:mavelasquezm@gmail.com}{mavelasquezm@gmail.com}}

\date{February 2022}

\begin{document}

\onehalfspacing
\maketitle

\begin{abstract}
We present a decomposition of rational twisted $G$-equivariant K-theory, $G$ a finite group, into cyclic group equivariant K-theory groups of fixed point spaces. This generalises the untwisted decomposition by Atiyah and Segal \cite{AtiyahSegal:decomp} as well as the decomposition by Adem and Ruan for twists coming from group cocycles \cite{AdemRuan:twisted}.
\end{abstract}

\section{Introduction}\label{sec:introduction}

Atiyah and Segal proved that the equivariant K-theory of a compact space $X$ acted on
by a finite group $G$ can, after tensoring with $\C$, be decomposed into the
non-equivariant K-theory of its fixed point spaces \cite{AtiyahSegal:decomp}*{Theorem 2}. That is, there is a natural isomorphism
\[
K_G(X) \otimes \C 
\xrightarrow{\cong}
\Bigl[ \bigoplus_{g \in G} K(X^g) \otimes \C \Bigr]^G.
\]
In this paper, we generalise this to twisted equivariant K-theory for twists
that are classified up to isomorphism by $H^3_G(X;\Z)$. For twists coming only
from the group, that is, with characteristic class in the image of $H^3(G;\Z)=H^3_G(*;\Z) \to H^3_G(X;\Z)$, Atiyah and Segal's decomposition has already been generalised to twisted K-theory by Adem and Ruan \cite{AdemRuan:twisted}*{Theorem 7.4}. Our decomposition map is defined by restriction to the fixed point spaces $X^g$; we have
\[
    K_G(X, P) \otimes \Q 
    \to
    \Bigl[ \bigoplus_{g \in G} K_{\cg g}(X^g, P|_{X^g}) \otimes \Q \Bigr]^G,
\]
where $K_G(X,P)$ denotes the $G$-equivariant K-theory of $X$ twisted by an equivariant principal $PU(\H)$-bundle $P$. Our main theorem is that this is an isomorphism onto a subspace defined by a simple relation between the summands required by naturality of the restriction maps. This is described in detail in Section \ref{sec:decomptheorem}. The twisted (equivariant) K-theory is by definition the K-theory of an appropriate non-commutative $C^*$-algebra, as explained in Section \ref{sec:def_twisted_K}.

Our result has to be distinguished from the Atiyah-Segal \emph{completion} theorem which describes equivariant K-theory completed at the augmentation ideal as the (representable) non-equivariant K-theory of the Borel construction. This has been generalized to K-theory of $C^*$-algebras and in particular to twisted equivariant K-theory in \cite{Phillips}. Some of our decomposition results also hold for the equivariant K-theory of an arbitrary $G$-$C^*$-algebra; these are explained in Section \ref{sec:rest-ind}  and might be useful in this generality.

The interest in twisted K-theory in mathematics is largely a result of an inflow of ideas from physics over the past few decades. It was defined early in the days of K-theory; initially for torsion twists by Donovan-Karoubi \cite{DonovanKaroubi} and then later for general twists by Rosenberg \cite{Rosenberg:TKT}. Twisted K-theory is of interest to physicists due to its applications in string theory, in particular as the home of D-brane charges in the presence of B-fields \cite{Witten:DbranesKtheory}. In this context, it also shows up naturally when describing T-duality \cites{BEM:Tduality,BunkeSchick}. The significance of twisted equivariant K-theory in particular is highlighted by the results of Freed, Hopkins, and Teleman \cite{FHT1}, who prove a close relationship between the twisted equivariant K-theory of a compact Lie group with its conjugation action and the Verlinde algebra of its loop group. 

We hope that our Atiyah-Segal type decomposition theorem can shed some light on T-duality transforms in equivariant K-theory.

\section{Twisted K-Theory}\label{sec:twistedktheory}
\label{sec:def_twisted_K}

Here we establish the definition of twisted equivariant K-theory used in this paper. Let $X$ be a compact space acted on by a finite group $G$ and let $P \to X$ be a stable $G$-equivariant principal $PU(\H)$-bundle. For simplicity, we call these bundles $G$-equivariant twists. Stable equivariant projective unitary bundles are defined, for instance, in \cite{BEJU:universaltwist}*{Def 2.2}. The $P$-twisted $G$-equivariant K-theory of $X$ is defined as the $G$-equivariant K-theory of the $C^*$-algebra of continuous sections of the associated bundle of compact operators.

\begin{definition}
    \[
    K^*_G(X,P) := K^G_* \bigl( \Gamma(P \times_{PU(\H)} \mathcal{K}) \bigr)
    \]
\end{definition}

This definition is motivated by Rosenberg's definition of twisted K-theory \cite{Rosenberg:TKT}*{\textsection 2}. The equivariant version appears in, for instance, \cite{Karoubi:oldandnew}*{\textsection 5.4} and \cite{Meinrenken}*{\textsection 2.3}.

There are of course other formulations of twisted equivariant K-theory; for example, via equivariant sections of the bundle of Fredholm operators associated to $P$ \cite{AtiyahSegal:TKT}*{\textsection 7}. The formulation we use allows us to work with twisted equivariant K-theory using techniques from non-commutative geometry. Indeed, by the Green-Julg theorem, twisted equivariant K-theory simply becomes the ordinary K-theory of some $C^*$-algebra, namely the crossed product of $G$ with the above algebra of sections. For this reason, in Section \ref{sec:rest-ind} and Section \ref{sec:K>} we will work with the equivariant K-theory of $C^*$-algebras and prove decomposition results that hold in this general context.

\section{The Decomposition Theorem}\label{sec:decomptheorem}

Let us formulate the decomposition theorem in detail. For each $g \in G$, the inclusion $X^g \to X$ of the fixed point set $X^g = \{ x \in X \mid g \cdot x = x \}$ induces a map
\[
    K_{G}(X, P) \to K_{\cg{g}}(X^g, P_g)
\] 
by first restricting the group from $G$ to $\cg g$ and then restricting to $X^g$. Here, $P_g$ is the restriction of $P$ to $X^g$. Applying this to each $g \in G$ gives 
\begin{equation}\label{eq:rest_to_fixpts}
    K_G(X, P) \to \bigoplus_{g \in G} K_{\cg{g}}(X^g, P_g).
\end{equation}
The inclusion of cyclic subgroups $\cg{h} \subset \cg{g}$ induces the following commutative diagram:
\[
\begin{tikzcd}
    K_G(X, P)  \arrow[r] \arrow[d] &
    K_{\cg{g}}(X^g, P_g)  \arrow[d] \\
    K_{\cg h}(X^{h}, P_{h}) \arrow[r] &
    K_{\cg{h}}(X^g, P_g).
\end{tikzcd}
\]
Thus there is a relation in the image of \eqref{eq:rest_to_fixpts}: for each $g,h \in G$ with $\cg h \subseteq \cg g$, the factors in the $g$- and $h$-summands map to the same element in $K_{\cg h}(X^g, P_g)$.

A $G$-action on the right-hand side of \eqref{eq:rest_to_fixpts} is defined as follows. The action of $k \in G$ induces a homeomorphism $k \colon X^g \to X^{kgk^{-1}}$. We then obtain the composition of isomorphisms
\begin{equation}\label{eq:twist_action}
    K_{\cg{kgk^{-1}}}(X^{kgk^{-1}}, P_{kgk^{-1}})
    \xrightarrow{\,\, \cong \,\,}
    K_{\cg g}(X^g, k^*P_{kgk^{-1}}) 
    \xrightarrow{\,\, \cong \,\,}
    K_{\cg g}(X^g, P_g).
\end{equation}
The second isomorphism is obtained by the canonical identification $k^*P_{kgk^{-1}} \cong P_g$ given by the action of $k$ on $P$. In this way, we get an action of $G$ on the direct sum of all the $K_{\cg g}(X^g, P_g)$. Moreover, the image of \eqref{eq:rest_to_fixpts} takes values in the $G$-invariants due to the following commutative diagram:
\[
\begin{tikzcd} 
    K_G(X, P)  \arrow[r] \arrow[d, "k", swap] &
    K_{\cg{kgk^{-1}}}(X^{kgk^{-1}}, P_{kgk^{-1}}) \arrow[d, "k"] \\
    K_G(X,P) \arrow[r] &
    K_{\cg g}(X^g, P_g)  
\end{tikzcd}
\]
Our main theorem states that, after tensoring with the rationals, these two
conditions on the image of \eqref{eq:rest_to_fixpts} are the right ones to produce an isomorphism. From now on we write $(-)_\Q = (-) \otimes \Q$, and similarly for $\C$.

\begin{theorem}
Let $G$ be a finite group, $X$  a finite $G$-CW-complex, and $P$ a $G$-equivariant twist on $X$. Then, there is an isomorphism
    \[
    K_G(X, P)_\Q 
    \xrightarrow{\cong} CSC_G(X,P)\subseteq
    \Bigl[ \bigoplus_{g \in G} K_{\cg g}(X^g, P_g)_\Q \Bigr]^G
    \]
onto the subspace $CSC_G(X,P)$ defined by the following relation:
\begin{center}
    If $g,h \in G$ and $\cg h \subseteq \cg g$, then the $g$-summand and the $h$-summand map to the same element in $K_{\cg h}(X^g, P_g)_\Q$.
\end{center}
We shall call this the cyclic subgroup compatibility condition.
\end{theorem}

\begin{remark}
The subspace $CSC_G(X,P)$ can be described as a limit over all the spaces $K_{\cg h}(X^g, P_g)_\Q$ with $\cg h \subseteq \cg g$ and arrows coming from the cyclic subgroup compatibility relation.
\end{remark}

The theorem is proved in the standard way. First, we must show that the theorem holds for homogeneous spaces $G/H$; for this, we will compare our formulation to Adem and Ruan's. Then we want to use Mayer-Vietoris to do induction on the $G$-cells of $X$. The difficulty is showing that, after imposing the cyclic subgroup compatibility condition, the Mayer-Vietoris sequence is still exact. The full proof is in Section \ref{sec:decompproof}.

\subsection{Example: Klein four group acting on a point}

Let $G = \Z_2 \times \Z_2$ act trivially on a point. In this case, since $H^3(\Z_2 \times \Z_2;\Z) \cong \Z_2$, there is one non-trivial twist up to isomorphism. Let $\tau$ be a non-trivial twist; the restriction $\tau_g$ is isomorphic to the trivial twist for all $g \in G$. We have one summand for each element of $G$, and each of these summands is preserved by the $G$-action because the group is abelian. In the following, all K-theory groups are concentrated in degree $0$. For each non-identity element of $G$, we have a copy of $\Q \oplus \Q$ in  $\bigoplus_{g \in G} K_{\cg g}(*, \tau_g)_\Q$ because $K_{\cg g}(*, \tau_g)_\Q \cong \Q \oplus \Q$ when $g$ has order 2. 
\[
\begin{array}{|c|c|c|c|c|}
    \hline
    g   & (0,0) & (1,0) & (0,1) & (1,1) \\
    \hline
    K_{\cg g}(*, \tau_g)_\Q & \Q & \Q \oplus \Q & \Q \oplus \Q & \Q \oplus \Q \\
    \hline
\end{array}
\]

The action of $g \in G \setminus \{e\}$ on $K_{\cg g}(*, \tau_g)_\Q$ is non-trivial even though the induced map $g^*$ on the space is the identity and $\tau_g$ is isomorphic to the trivial twist. The action is described in \eqref{eq:twist_action}; we get
\[
    K_{\cg g}(*, \tau_g) 
    \xrightarrow{\,\, g^* \,\,}
    K_{\cg g}(*, \tau_g) 
    \xrightarrow{\,\, \cdot g^{-1} \,\,}
    K_{\cg g}(*, \tau_g).
\]
The first map is the identity. The second map is induced by a twist automorphism coming from multiplication by $g^{-1} = g$. This automorphism is the one corresponding to the non-trivial element of $H^2(\Z_2;\Z) \cong \Z_2$. This exchanges the two factors in $K_{\cg g}(*, \tau_g)_\Q \cong \Q \oplus \Q$. Therefore the $G$-invariant subspace of each $\Q \oplus \Q$ is precisely the diagonal $\Delta \Q \subset \Q \oplus \Q$. 

For each $g \in G$, we have that $\cg e \subset g$, and so the cyclic subgroup relation implies that for each of the $(\Q \oplus \Q)$-summands, one of the factors is determined by the $e$-summand $K(*)_\Q \cong \Q$. Together with the $G$-invariance, we conclude that all the summands are determined by the $e$-summand; hence, $K_G(*,\tau)_\Q \cong \Q$. This example also appears in \cite{AdemRuan:twisted}*{Example 7.8}.

\subsection{Example: \texorpdfstring{$D_8$ acting trivially on $S^1$}{D8 acting trivially on S1}}

We calculate the twisted equivariant $K^0$-groups of $S^1$ with trivial $D_8$ action, where $D_8 = \langle r,s \mid r^4 = s^2 = e, srs = r^3 \rangle $ is the dihedral group of order 8. The twists are classified by 
\begin{align*}
    H^3_{D_8}(S^1;\Z) 
    &\cong
    H^3(S^1 \times BD_8;\Z)  \\
    &\cong
    H^3(D_8;\Z) \oplus H^2(D_8;\Z) \\
    &\cong
    \Z_2 \oplus (\Z_2 \oplus \Z_2).
\end{align*}
We only consider the twists coming from $H^2(D_8;\Z) \cong \Z_2 \oplus \Z_2$. One can use Adem and Ruan's formula for twists coming from $H^3(D_8;\Z)$. Using the isomorphism $H^2(D_8;\Z) \cong H^1(D_8;S^1)$, the twists we consider are induced from 1-cocycles $D_8 \to S^1$. Such maps must factor through the abelianisation $D_8 / \cg{r^2} \cong \Z_2 \times \Z_2$. The following describes the four possibilites, denoted $\tau_1, \tau_2, \tau_3$, and $\tau_4$:
\[
\begin{array}{|c|c|c|c|c|}
\hline
    & 1,r^2 &  r, r^3   & s, r^2s   & rs, r^3s \\
\hline
\tau_1  & 1 & 1 & 1 & 1 \\
\tau_2 & 1 & -1 & 1 & -1 \\
\tau_3 & 1 & 1 & -1 & -1 \\
\tau_4 & 1 & -1 & -1 & 1 \\
\hline
\end{array}
\]
The restrictions to the cyclic subgroups of $D_8$ can be read from the table; a twist restricts to a non-trivial twist whenever there is a $-1$. Since $H^2(\Z_2;\Z) \cong \Z_2$ there is, up to isomorphism, only one non-trivial twist on the subgroups of order 2. The non-trivial twist on the $r$-summand corresponds to the order 2 element of $H^2(\Z_4;\Z) \cong \Z_4$. Using a Mayer-Vietoris argument, one can compute that
\[
K^0_{\Z_2}(S^1, \xi_1)_\Q \cong \Q
\quad \text{and} \quad
K^0_{\Z_4}(S^1, \xi_2)_\Q \cong \Q^2,
\]
where $\xi_1$ and $\xi_2$ are these non-trivial $\Z_2$- and $\Z_4$-equivariant twists, respectively. The argument is the same as the computation in \cite{FHT1}*{Example 1.6}. Here, one can actually compute the integral twisted equivariant K-theory, not just the rationalization. The untwisted K-theory is $K^0_{\Z_2}(S^1) \cong R(\Z_2)$ and $K^0_{\Z_4}(S^1) \cong R(\Z_4)$, the complex representation rings.

Let $\tau$ be one of the twists $\tau_1,\tau_2, \tau_3$, or $\tau_4$. We want to determine
\[
    CSC_{D_8}(S^1, \tau) 
    \subseteq
    \Bigl[ \bigoplus_{g \in D_8} K^0_{\cg g}(S^1, \tau_g)_\Q \Bigr]^{D_8}
\]
We summarise the contribution of each of the group elements:
\[
\begin{array}{|c|c|c|c|c|c|}
    \hline
 & e & r^2 & r & s & rs \\
\hline 
 \tau_1 & \Q & \Q^2 & \Q^4 & \Q^2 & \Q^2 \\ 
 \tau_2 & \Q & \Q^2 & \Q^2 & \Q^2 & \Q \\
 \tau_3 & \Q & \Q^2 & \Q^4 & \Q  & \Q \\
 \tau_4 & \Q & \Q^2 & \Q^2 & \Q & \Q^2 \\
\hline
\end{array}
\]
Note that we only need to write one element of each conjugacy class, as summands corresponding to conjugate elements are identified by the group action. 

To see the conditions imposed by the cyclic subgroup condition, we investigate the relations between the cyclic subgroups. These are shown in the following diagram:
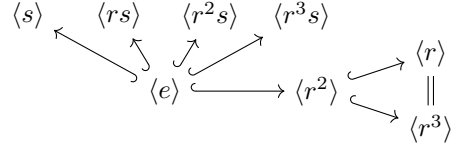
\begin{center}
\begin{tikzpicture}
    \node (e) at (0,0) {$\cg e$};
    \node (s) at (-1.8,1) {$\cg s$};
    \node (rs) at (-0.6,1) {$\cg{rs}$};
    \node (r2s) at (0.6,1) {$\cg{r^2s}$};
    \node (r3s) at (1.8,1) {$\cg{r^3s}$};
    \node (r2) at (2,0) {$\cg{r^2}$};
    \node (r) at (3.5,0.5) {$\cg r$};
    \node (r3) at (3.5,-0.5) {$\cg{r^3}$};
    
    \draw [right hook->] (e) -> (s);
    \draw [right hook->] (e) -> (rs);
    \draw [right hook->] (e) -> (r2s);
    \draw [right hook->] (e) -> (r3s);
    \draw [right hook->] (e) -> (r2);
    \draw [right hook->] (r2) -> (r);
    \draw [right hook->] (r2) -> (r3);  
    
    \draw ([xshift=-1pt] r.south) -- ([xshift=-1pt] r3.north);
    \draw ([xshift=1pt] r.south) -- ([xshift=1pt] r3.north);
\end{tikzpicture}
\end{center}
One can check that the maps induced by restriction to subgroups are all injective. The trivial subgroup is a subgroup of every cyclic group, so the $e$-summand determines one factor of each of the other summands. We also have $\cg{r^2} \subset \cg{r}$, so two factors of the $r$-summand are determined by the $r^2$-summand. The final relation is that $\cg{r} = \cg{r^3}$. This means that the $\cg{r}$- and $\cg{r^3}$-summands are equal. The elements $r$ and $r^3$ are also related by conjugation by $s$, and when the twist $\tau_r$ is trivial this action induces an automorphism on $K^0_{\cg r}(S^1, \tau_r) \cong \Z_4$ that swaps the two order 4 elements. The result is that the invariant subspace of $K^0_{\cg r}(S^1, \tau)_\Q \cong \Q^4$ is of rank 3. We can now compute the rational twisted equivariant K-theory by counting dimensions:
\[
\begin{array}{|c|c|}
\hline 
\tau & K^0_{D_8}(S^1, \tau)_\Q \\
\hline 
\tau_1 & \Q^5 \\
\tau_2 & \Q^3 \\
\tau_3 & \Q^3 \\
\tau_4 & \Q^3 \\
\hline
\end{array}
\]
As a sanity check, we know that $K^0_{D_8}(S^1, \tau_1) \cong R(D_8)$, which is of rank $5$ because $D_8$ has 5 conjugacy classes. We remark that this example can be computed integrally using the Mayer-Vietoris technique in \cite{FHT1}*{Example 1.6}; our discussion serves as a demonstration of the decomposition theorem.

\subsection{Comparison with Atiyah-Segal}

Before proceeding, we explain how our result is a generalisation of the decompositions of Atiyah-Segal and Adem-Ruan. The Atiyah-Segal map,
\[
K_G(X)_\C \to \Bigl[ \bigoplus_{g \in G} K(X^g)_\C \Bigr]^G,
\]
is defined as a direct sum of maps
\[
K_G(X)_\C 
\to
K_{\cg g}(X^g)_\C 
\cong
K(X^g) \otimes R(\cg g)_\C
\to
K(X^g)_\C,
\]
where the isomorphism is because $\cg g$ acts trivially on $X^g$ and the final map is induced by sending a character $\phi$ to $\phi(g)$. The resulting map factors through the map in our decomposition theorem (using $\C$ instead of $\Q$),
\begin{equation}\label{eqn:atiyahsegalcomp}
K_G(X)_\C 
\to \Bigl[ \bigoplus_{g \in G} K_{\cg g}(X^g)_\C \Bigr]^G
\to \Bigl[ \bigoplus_{g \in G} K(X^g)_\C \Bigr]^G.
\end{equation}
Given an element in the right-most space, we recover the $g$-summand in the middle space as follows. For each $k\in \{0,\dotsc, |g|-1\}$, the element in the $g^k$-summand of the Atiyah-Segal space is a sum $F_k = \sum_{i} E_i \xi^{ik}$, where $E$ is a $G$-vector bundle on $X$, $E_i$ is the $g^i$-isotopic component of $E|_{X^g}$, and $\xi$ is the $|g|$th root of unity. By appropriately weighting each term by powers of $\xi$, we can recover each $E_i$ by adding together the $F_k$. This is essentially an inverse discrete Fourier transform. The $g$-summand in the middle space is then just $\sum_{i} E_i \otimes \chi_i$, where $\chi_i$ is the representation $\chi_i(g) = \xi^i$. Repeating this for every summand gives us a split of the second map in \eqref{eqn:atiyahsegalcomp}. The image is precisely the cyclic subgroup compatible elements, giving us an isomorphism between our decomposition and that of Atiyah and Segal. If we use twisted characters, then we can similarly recover the decomposition theorem of Adem and Ruan, as explained in detail in the next subsection. Note that, to use the Fourier decomposition of Atiyah-Segal or Adem-Ruan, one is forced to work with $\C$ instead of $\Q$.

\subsection{Comparison with Adem-Ruan}\label{subsec:AdemRuan}

For a twist $P$ classified by an element in the image of $H^3(G;\Z) \to H^3_G(X;\Z)$, that is, represented by a $\C^*$-valued group 2-cocycle $\alpha$, we are in the context of Adem and Ruan's paper \cite{AdemRuan:twisted}. Their decomposition also factors through ours:
\begin{equation}\label{eqn:ademruandecomp}
    K_G(X, P)_\C 
    \to \Bigl[ \bigoplus_{g \in G} K_{\cg g}(X^g, P_g)_\C \Bigr]^G
    \to \bigoplus_{[g]} \bigl[ K(X^g, P_g) \otimes L_g \bigr]^{C_g}.
\end{equation}
Here $L_g$ is a one-dimensional representation of the centraliser $C_g$ of $g$ defined by the map $h \mapsto \alpha(h,g)\alpha(g,h)^{-1}$. The final direct sum is defined over the conjugacy classes of $G$; a representative of each conjugacy class is chosen. The second map is given by the following composition:
\[
    K_{\cg g}(X^g, P_g)_\C 
    \cong K(X^g) \otimes R_{\operatorname{res}(\alpha)}(\cg g)_\C 
    \to K(X^g) \otimes L_g.
\]
To explain:
\begin{itemize}
    \item $R_{\operatorname{res}(\alpha)}(\cg g)$ is the ring of $\operatorname{res}(\alpha)$-twisted characters of $\cg g$, where $\operatorname{res}(\alpha)$ is the restriction of $\alpha$ to $\cg g$.
    \item The isomorphism exists because $\cg g$ acts trivially on $X^g$ and the twist comes only from the group; see \cite{AdemRuan:twisted}*{Lemma 7.3}.
    \item The second map is given by evaluating twisted characters at $g$, that is, $\chi \mapsto \chi(g)$.
\end{itemize}
A splitting of the second map in \eqref{eqn:ademruandecomp} can be constructed in the same way as in the Atiyah-Segal case, except $\alpha$-twisted characters are used. Since $H^3(\cg g;\Z) = 0$, we know that $\operatorname{res}(\alpha) = \delta \beta$ where $\delta \beta$ is the boundary map of cocycles applied to a 1-cochain $\beta \colon \cg g \to \C^*$. Every $\alpha$-twisted character is of the form $\beta \cdot \chi$ for $\chi$ an untwisted character of $\cg g$. One now performs the same calculation as the previous section, inserting $\beta$ in the relevant places. We apply the group action to get the summands corresponding to elements that aren't one of the chosen conjugacy class representatives.

\section{Restriction and Induction Maps}\label{sec:rest-ind}

For this section, let $G$ be a finite abelian group and $A$ a $G$-$C^*$-algebra.  We emphasize that this $G$ is not the same $G$ as in the decomposition theorem; rather, it will be one of the finite cyclic groups $\cg g$. In our discussion there will be two important maps. Let $H \subseteq G$ be a subgroup. The two maps, to be defined, are
\[
i_H \colon K^H(A) \to K^G(A)
\quad \text{and} \quad
r_H \colon K^G(A) \to K^H(A).
\]
The first is simple to define. By the Green-Julg theorem, we can identify equivariant K-theory with the K-theory of crossed products. The map $K(A \rtimes H) \to K(A \rtimes G)$ corresponding to $i_H$ is induced by the inclusion of $A \rtimes H$ into $A \rtimes G$.

The restriction map $r_H$ is defined as follows. First, suppose that $A$ is unital. Then $A \rtimes G$ is also unital and there is a canonical isomorphism $A \rtimes G = \End_{A \rtimes G}(A \rtimes G)$, where $A \rtimes G$ is viewed as a left-$(A \rtimes G)$-module. We have the following maps:
\[
A \rtimes G 
= \End_{A \rtimes G}(A \rtimes G)
\hookrightarrow \End_{A \rtimes H}(A \rtimes G)
\leftarrow A \rtimes H
\]
The final map sends $f_0 \in A \rtimes H$ to the $(A \rtimes H)$-module map 
\[
\sum_{g\in G} a_g g 
\longmapsto 
\Bigl( \sum_{h \in H} a_h h \Bigr) \cdot f_0.
\]
This function is isomorphic to the standard block inclusion $A\rtimes H\hookrightarrow M_{[G:H]}(A\rtimes H)$ after identifying the matrix algebra with $\End_{A \rtimes H}(A \rtimes G)$, and hence induces an isomorphism on K-theory. Its inverse allows us to define $r_H$ as the composition
\[
r_H \colon 
K \bigl(A \rtimes G \bigr)
= K \bigl(\End_{A \rtimes G}(A \rtimes G) \bigr)
\to K \bigl( \End_{A \rtimes H}(A \rtimes G) \bigr)
\rightarrow 
K \bigl(A \rtimes H \bigr).
\]
When $A$ is non-unital, one defines the map for the unitalisation $A_+ = A \oplus \C$ and then restricts to the K-theory of $A$.

Since $H$ is normal in $G$, there is a $G$-action on $A \rtimes G$ and $A \rtimes H$ given by conjugation by $G$ considered as unitary elements in $A\rtimes G$ . Indeed, on $A\rtimes G$ this is the inner action. Hence the induced action on $K^G(A) = K(A \rtimes G)$ is trivial, but the action on $K^H(A) = K(A \rtimes H)$ may not be. For $x \in A \rtimes H$ and $g \in G$ we write $x^g$ for the element $x$ acted on by $g$. The maps used to define $r_H$ are all $G$-equivariant (after defining suitable actions on the endomorphism algebras) and $K^G(A)^G = K^G(A)$, so $r_H$ takes values in $K^H(A)^G$, the $G$-invariant component of $K^H(A)$.

\begin{lemma}\label{lemma:functoriality}
$i_H$ and $r_H$ satisfy the following properties:

\begin{enumerate}[label=(\alph*)]
    \item If $H_1 \subseteq H_2 \subseteq G$, then the following commute:
        \[
        \begin{tikzcd}
            K^{H_1}(A) \arrow[r, "{i_{H_1}}"] \arrow[rd, "{i_{H_1}}", swap]
            & K^{H_2}(A) \arrow[d, "{i_{H_2}}"] \\
            & K^G(A)
        \end{tikzcd}
        \qquad
        \begin{tikzcd}
            K^G(A) \arrow[r, "{r_{H_2}}"] \arrow[rd, "{r_{H_1}}", swap]
            & K^{H_2}(A) \arrow[d, "{r_{H_2}}"] \\
            & K^{H_1}(A)
        \end{tikzcd}
        \]
    
    \item If $H \subseteq G$ and $f \colon A \to B$ is a $G$-equivariant morphism, then the following diagrams commute:
        \[
        \begin{tikzcd}
            K^H(A) \arrow[r, "i_H"] \arrow[d, "f_*", swap] 
            & K^G(A) \arrow[d, "f_*"] \\
            K^H(B) \arrow[r, "i_H"] 
            & K^G(B)
        \end{tikzcd}
        \qquad
        \begin{tikzcd}
            K^G(A) \arrow[r, "r_H"] \arrow[d, "f_*", swap] 
            & K^H(A) \arrow[d, "f_*"] \\
            K^G(B) \arrow[r, "r_H"] 
            & K^H(B)
        \end{tikzcd}
        \]
    \item If $0 \to J \to A \to A/J \to 0$ is an exact sequence of $G$-$C^*$-algebras, then $i_H$ and $r_H$ commute with the index maps $\partial \colon K^*(A/J) \to K^{*-1}(J)$, that is, the following commute:
        \[
        \begin{tikzcd}
            K^H_*(A/J) \arrow[r, "i_H"] \arrow[d, "\partial", swap] 
            & K^G_*(A/J) \arrow[d, "\partial"] \\
            K^H_{*-1}(J) \arrow[r, "i_H"] 
            & K^G_{*-1}(J)
        \end{tikzcd}
        \qquad
        \begin{tikzcd}
            K^G_*(A/J) \arrow[r, "r_H"] \arrow[d, "\partial", swap] 
            & K^H_*(A/J) \arrow[d, "\partial"] \\
            K^G_{*-1}(J) \arrow[r, "r_H"] 
            & K^H_{*-1}(J)
        \end{tikzcd}
        \]
    As a consequence, the K-theory long exact sequence and Mayer-Vietoris sequence are natural with respect to the maps $i_H$ and $r_H$.
\end{enumerate}

\end{lemma}

\begin{proof}
$i_H$ and $r_H$ are defined using the functoriality of K-theory for certain maps between naturally constructed $C^*$-algebras. Proving these properties comes down to writing the relevant diagrams of $C^*$-algebras and showing that the maps commute. We leave the details to the reader.
\end{proof}

\begin{lemma}\label{lemma:roi}
    The composition $r_H \circ i_H \colon K^{H}(A)_{\Q} \to K^{H}(A)_{\Q}$ is given by 
    \[
        r_H \circ i_H(x) =
        \frac{[G : H]}{|G|} \sum_{g \in G} g \cdot x.
    \]
In other words, $r_H \circ i_H$ is $[G :H]$ times the averaging map. In particular, $r_H \circ i_H$ is multiplication by $[G:H]$ when restricted to $K^H(A)^G$.
\end{lemma}

\begin{proof}
The map $i_H \colon K^H(A)_\Q \to K^G(A)_\Q$ takes values in the $G$-invariants because $G$ acts trivially on $K^G(A)$. Therefore $i_H$ factors through $K^H(A)^G_{\Q}$ via the averaging map. This is where taking the tensor product with $\Q$ is necessary. It now suffices to show that $r_H \circ i_H$ restricted to $K^H(A)^G_\Q$ is multiplication by $[G:H]$.

Assume that $A$ is unital; it is sufficient to prove the theorem in this case. Choosing representatives $g_i$ for elements of $G/H$, one has an isomorphism 
\[
\End_{A \rtimes H}(A \rtimes G) \cong M_{[G:H]}(A \rtimes H).
\]
The induced isomorphism on K-theory does not depend on the choice of representatives. The composition
\[
    A \rtimes H 
    \hookrightarrow
    A \rtimes G
    \hookrightarrow
    \End_{A \rtimes H}(A \rtimes G) 
    \cong 
    M_{[G:H]}(A \rtimes H)
\]
is given by
\[
    x \mapsto 
    \begin{bmatrix}
        x^{g_1} & 0 && \\
        0 & x^{g_2} && \\
        && \ddots & \\
        &&& x^{g_{[G:H]}}
    \end{bmatrix}.
\]
The induced map on K-theory is
\[
    K^H(A) \to K^H(A), 
    \quad
    x \mapsto \sum_{i=1}^{[G:H]} x^{g_i}.
\]
After restricting to $G$-invariants, this is multiplication by $[G:H]$, as required.
\end{proof}

\begin{lemma}\label{lemma:rest-ind}
Let $H_1$ and $H_2$ be subgroups of $G$. The following diagram commutes up to multiplication by $[G:H_1H_2]$:
\begin{equation}\label{diag:rest-ind}
\begin{tikzcd}
    K^{H_1}(A)^G \arrow[r, "i_{H_1}"] \arrow[d, "r_{H_1 \cap H_2}", swap]
        & K^G(A) \arrow[d, "r_{H_2}"]
    \\
    K^{H_1 \cap H_2}(A)^G \arrow[r, "i_{H_1 \cap H_2}"] 
        & K^{H_2}(A)^G
\end{tikzcd}
\end{equation}
\end{lemma}

\begin{proof}
First, we prove the result when $G = H_1H_2$. Consider the following diagram:
\begin{equation}\label{diag:rest-indprod}
\begin{tikzcd}
    A \rtimes H_1  \arrow[d] \arrow[r]
    &   A \rtimes H_1H_2 \arrow[d] \\
    \End_{A \rtimes H_1 \cap H_2}(A \rtimes H_1) \arrow[r]
    & \End_{A \rtimes H_2}(A \rtimes H_1H_2) \\
    A \rtimes H_1 \cap H_2 \arrow[r] \arrow[u]
    & A \rtimes H_2 \arrow[u]
\end{tikzcd}
\end{equation}
The map in the middle row is defined as follows. Let $f \colon A \rtimes H_1 \to A \rtimes H_1$ be an $(A \rtimes H_1 \cap H_2)$-module morphism and $x = \sum_{g \in H_1H_2} a_g g \in A \rtimes H_1H_2$. For each $h_2 \in H_2$ we set $x_{h_2} := \sum_{h_1 \in H_1} a_{h_1h_2} h_1 \in A \rtimes H_1$. Applying $f$ to $x_{h_2}$ gives another element, say $f(x_{h_2}) = \sum_{h_1 \in H_1} \tilde a_{h_1,h_2} h_1$. Thus, given $f \colon A \rtimes H_1 \to A \rtimes H_1$, we define an endomorphism of $A \rtimes H_1H_2$ by
\[
x = \sum_{g \in H_1H_2} a_g g 
\longmapsto
\sum_{h_1h_2 \in H_1H_2} \tilde a_{h_1, h_2} h_1h_2.
\]
One must check that if $h_1h_2 = h_1'h_2' \in H_1H_2$ then $\tilde a_{h_1, h_2} = \tilde a_{h_1', h_2'}$. This follows from the fact that $f$ is an $(A \rtimes H_1 \cap H_2)$-module map. One also checks that the resulting map is an $(A \rtimes H_2)$-module map.

Diagram \eqref{diag:rest-indprod} commutes, and the two outer paths from $A \rtimes H_1$ to $A \rtimes H_2$ define the two maps from $K^{H_1}(A)^G$ to $K^{H_2}(A)^G$ described in the theorem. Thus, the result holds for $G = H_1H_2$.

The general case is implied by the case $G = H_1H_2$. Consider the following diagram, which describes the composition $K^{H_1}(A)^G \to K^G(A) \to K^{H_2}(A)^G$:
\[
\begin{tikzcd}
    & & K^G(A) \arrow[rd, "r_{H_1H_2}", swap] \arrow[rrd, "r_{H_2}"] & & \\
    K^{H_1}(A)^G \arrow[rru, "i_{H_1}"] \arrow[r, "i_{H_1}", swap] 
    & K^{H_1H_2}(A)^G \arrow[ru, "i_{H_1H_2}", swap] \arrow[rr, "{\cdot [G:H_1H_2]}", swap] 
    && K^{H_1H_2}(A)^G \arrow[r, "r_{H_2}", swap] 
    & K^{H_2}(A)^G
\end{tikzcd}
\]
The commutativity of the left and right triangles follows from Lemma \ref{lemma:functoriality} and commutativity of the center triangle is a result of the Lemma \ref{lemma:roi}. The composition $K^{H_1}(A)^G \to K^{H_1 \cap H_2}(A)^G \to K^{H_2}(A)^G$ remains unchanged, so we conclude that diagram \eqref{diag:rest-indprod} commutes up to multiplication by $[G:H_1H_2]$.
\end{proof}

\section{The Subgroup Independent Component of \texorpdfstring{$K^G(A)$}{K G(A)}}\label{sec:K>}

Let $G$ be a finite abelian group in this section as well. The collection of maps $r_H \colon K^G(A) \to K^H(A)^G$ give rise to a map $K^G(A) \to \bigoplus_{H \subsetneq G} K^H(A)^G$. Consider the shared kernel of all of these maps:

\begin{definition}
    \[
    K^G_{>}(A) := \ker \Bigl( K^G(A) \to \bigoplus_{H \subsetneq G} K^H(A)^G \Bigr)
    \]
\end{definition}

$K^G_>(A)$ is the component of $K^G(A)$ that doesn't depend on any subgroups of $G$. As an example, if $A = \C$, then $K^G(A) = R(G)$ and $K^G_>(A)$ is the subspace of characters $G \to \C$ that vanish on all proper subgroups of $G$. 

Note that, by definition, the restriction of $r_H \colon K^G(A) \to K^H(A)$ to $K^G_>(A)$ is the zero map. Lemma \ref{lemma:rest-ind} then implies that for $H_1, H_2 \subsetneq G$ and $x \in K^{H_1}_>(A)^G$,
\[
r_{H_2} \circ i_{H_1}(x)
=
\begin{cases}
    [G : H_1] \cdot x& H_1 = H_2,\\ 
    0, & H_1 \neq H_2.
\end{cases}
\]

\begin{lemma}\label{lemma:K>alt}
    There is a canonical isomorphism
\[
K^G_>(A)_\Q 
\cong 
K^G(A)_\Q  \,\, / \sum_{H \subsetneq G} i_H \bigl( K^H_>(A)^G_\Q \bigr).
\]
\end{lemma}
\begin{proof}
    The isomorphism is the quotient map restricted to $K^G_>(A)$. The inverse is
 \[
   K^G(A)_\Q  \,\, / \sum_{H \subsetneq G} i_H \bigl( K^H_>(A)^G_\Q \bigr)
   \longrightarrow
   K^G_>(A)_\Q,
   \quad
    [x] 
    \mapsto
    x - \sum_{H \subsetneq G} \frac{1}{[G:H]}i_H \circ r_H(x).
 \]
 This is well defined, since if $x_H \in K^H_>(A)$, then
 \begin{align*}
    \sum_{H \subsetneq G} i_H(x_H) 
    &\mapsto 
    \sum_{H \subsetneq G} i_H(x_H) - \sum_{H_1 \subsetneq G}  \sum_{H_2 \subsetneq G} \frac{1}{[G:H_1]} i_{H_1} \circ r_{H_1} \circ i_{H_2}(x_{H_2}) \\
    &= \sum_{H \subsetneq G} i_H(x_H) -\sum_{H \subsetneq G} i_H(x_H) \\
    &= 0,
\end{align*}
One easily checks that this is indeed the inverse.
\end{proof}

This lemma determines a preferred splitting projection onto $K^G_>(A)_\Q$ given by
\[
K^G(A)_\Q \to K^G_>(A)_\Q,
\qquad 
x \mapsto x - \sum_{H \subsetneq G} \frac{1}{[G:H]} i_H \circ r_H(x).
\]
The following lemma shows that the $K_>$-groups inherit exactness properties from the $K$-groups:

\begin{lemma}\label{lemma:liftto>}
Consider a map $K^G(A)_\Q \to K^G(B)_\Q$ and an element $x \in K^G_>(B)_\Q$. If there exists a lift of $x$ to $K^G(A)_\Q$, then there exists a lift of $x$ to $K^G_>(A)_\Q$. 
\end{lemma}

\begin{proof}
Let $\tilde x$ be a lift of $x$ to $K^G(A)$. For each $H \subsetneq G$, let $\tilde x^>_H$ the projection of $r_H(\tilde x)$ onto $K_H^>(A)$. Then $\tilde x - \sum_{H}\frac{1}{[G:H]} i_H(\tilde x^>_H)$ is a lift of $x$ to $K^G_>(A)$. \end{proof}

Now for the main theorem of this section:

\begin{theorem}\label{sestheorem}
There is the following split short exact sequence:
\begin{equation}\label{diag:sesK>}
    0 
    \longrightarrow
    \bigoplus_{H \subsetneq G} K^H_>(A)^G_\Q
    \longrightarrow
    K^G(A)_\Q 
    \longrightarrow
    K^G_>(A)_\Q 
    \to 
    0
\end{equation}
This is natural with respect to maps of $C^*$-algebras and boundary maps in the K-theory long exact sequence.
\end{theorem}

\begin{proof}
Lemma \ref{lemma:K>alt} gives the short exact sequence
\[
    0 
    \longrightarrow
    \sum_{H \subsetneq G} i_H(K^H_>(A)^G_\Q)
    \longrightarrow
    K^G(A)_\Q 
    \longrightarrow
    K^G_>(A)_\Q 
    \to 
    0. 
\]
We obtain \eqref{diag:sesK>} by noting that the $i_H$-maps give an isomorphism between the direct sum $\bigoplus_{H \subsetneq G} K^H_>(A)^G_\Q$ and $\sum_{H \subsetneq G} i_H(K^H_>(A)^G_\Q)$. The splitting can be described in two ways. A right split is given by the inclusion $K^G_>(A) \to K^G(A)$. A left split is given by the restriction maps $r_H$, followed by the projection map onto each $K^H_>(A)$. Naturality follows from the naturality of $i_H$ and $r_H$, proved in Lemma \ref{lemma:functoriality}.
\end{proof}

Since the $K_>$-groups are no easier to compute than the $K_G$-groups, this is not exactly a useful result for calculating $K^G(A)$. However, it is useful in our proof of the decomposition theorem because it allows us to isolate the component of $K^G(A)$ that depends on the subgroups of $G$ and the ``free'' component that doesn't depend on any of these subgroups.

\section{Proof of the Decomposition Theorem}\label{sec:decompproof}

Before beginning the proof, we restate the theorem.

\begin{theorem}
Let $G$ be a finite group, $X$  a finite $G$-CW-complex, and $P$ a $G$-equivariant twist on $X$. Then, there is an isomorphism
\begin{equation}\label{eqn:decompmap}
    K_G(X, P)_\Q 
    \xrightarrow{\cong} CSC_G(X,P)\subseteq
    \Bigl[ \bigoplus_{g \in G} K_{\cg g}(X^g, P_g)_\Q \Bigr]^G
\end{equation}
onto the subspace $CSC_G(X,P)$ defined by the following relation:
\begin{center}
    If $g,h \in G$ and $\cg h \subseteq \cg g$, then the $g$-summand and the $h$-summand map to the same element in $K_{\cg h}(X^g, P_g)_\Q$.
\end{center}
\end{theorem}

\begin{proof}

 First consider the homogeneous space $X = G/H$. This is discrete, so every $G$-equivariant twist comes from a group cocycle. The result is then true by comparison with the isomorphism of Adem and Ruan \cite{AdemRuan:twisted}; we compared our decomposition with Adem and Ruan's in \textsection \ref{subsec:AdemRuan}.

 Now, we wish to use induction on the $G$-cells of $X$ via a Mayer-Vietoris argument. For this, it is required that the functor $CSC_G$ satisfies the Mayer-Vietoris property. The direct sum in \eqref{eqn:decompmap} satisfies this property; we need to ensure that restricting to the cyclic subgroup compatible elements preserves it. Taking subspaces is left exact, so we already have that images are contained in kernels. It is then sufficient to show that if a cyclic subgroup compatible element lies in the kernel of a map in the Mayer-Vietoris sequence, then it can be lifted to an element that is also cyclic subgroup compatible. 

 Let $\bigl[ \bigoplus_{g \in G} K_{\cg g}(X^g, P_g)_\Q \bigr]^G \to \bigl[ \bigoplus_{g \in G} K_{\cg g}(Y^g, Q_g)_\Q \bigr]^G$ be a map in the Mayer-Vietoris sequence. It is the restriction of a direct sum of maps in the Mayer-Vietoris sequence for each $K_{\cg g}$. Let $y = (y_g)_{g \in G}$ be a cyclic subgroup compatible element in the kernel of the next map in the Mayer-Vietoris sequence. We construct a lift of this element to the cyclic subgroup compatible elements.

 Consider the partial order $<$ on $G$ where $h < g$ when $\cg h \subsetneq \cg g$. We construct a lift by induction on this poset. The smallest element in this poset is the identity, and $e < g$ for every $g \in G$. Start by choosing a lift $x_e$ of $y_e$ under the map $K(X,P)^G_\Q \to K(Y,Q)^G_\Q$. Now consider $g \in G$. Assume that for every $h < g$ we have compatible lifts $x_h$ of $y_h$ - compatible meaning that for every $h' < h$ the lifts satisfy the correct compatibility relation between the $h'$ and $h$-summands. Note that since the elements are $G$-invariant, they are also $\cg g$-invariant. By Theorem \ref{sestheorem} we have the following commutative diagram:
\[
    \begin{tikzcd}[column sep={2em},row sep=1em]
        0 \arrow[r] 
        & \bigoplus\limits_{h<g} K^>_{\cg h}(X^g, P_g)_\Q^{\cg g} \arrow[r] \arrow[d] 
        & K_{\cg g}(X^g, P_g)_\Q \arrow[r] \arrow[d] 
        & K^{>}_{\cg g}(X^g, P_g)_\Q \arrow[r] \arrow[d] 
        & 0 \\
        0 \arrow[r] 
        & \bigoplus\limits_{h<g} K^>_{\cg h}(Y^g, Q_g)_\Q^{\cg g} \arrow[r] 
        & K_{\cg g}(Y^g, Q_g)_\Q \arrow[r] 
        & K^{>}_{\cg g}(Y^g, Q_g)_\Q \arrow[r] 
        & 0 
    \end{tikzcd}
\]
The $K^>$-groups were defined in the previous section (we now write the $>$ as a superscript because we're using K-theory of spaces instead of algebras). We have been given elements in the bottom row and by assumption we have a lift $(x_h)_{h<g}$ of $(y_h)_{h<g}$ on the left-hand side of the diagram. Let $y'$ be the projection of $y$ onto $K^{>}_{\cg g}(Y^g, Q_g)_\Q$. By Lemma \ref{lemma:liftto>} there exists a lift $x' \in K^{>}_{\cg g}(X^g, P_g)_\Q$ of $y'$. Then $(x_h)_{h<g}$ and $x'$ together form an element in $K_{\cg g}(X^g, P_g)_\Q$ that is a lift of $y$ and satisfies cyclic subgroup compatibility. Thus, by induction, we can always construct the necessary lift, and the proof is complete.
\end{proof}

As a corollary, we can also decompose the twisted equivariant K-theory of a $G$-equivariant fiber bundle $E \to X$.  Note that this is precisely the situation encountered in T-duality, where the twisted K-theory of the total space of a $U(1)$-bundle or $T^n$-bundle has to be analyzed. We decompose the rational equivariant K-theory of $E$ into the twisted cyclic group equivariant K-theory of $E|_{X^g}$, the restriction of $E$ to the fixed point spaces of $X$. The maps are induced by the inclusions $E|_{X^g} \to E$, and the resulting map
\[
    K_G(E,P)_\Q 
    \to 
    \bigoplus_{g \in G} K_{\cg g}(E|_{X^g}, P|_{E|_{X^g}})_\Q
\]
has image the $G$-invariant, cyclic subgroup compatible elements. 

\begin{corollary}
Let $E \to X$ be a $G$-equivariant fiber bundle with fiber and base both $G$-CW-complexes. Let $P$ be a $G$-equivariant twist on $E$. Then, there is an isomorphism
\[
    K_G(E,P)_\Q 
    \to \widetilde{CSC}_G(E,P) 
    \subseteq 
    \Bigl[ \bigoplus_{g \in G} K_{\cg g}(E|_{X^g}, P|_{E|_{X^g}})_\Q \Bigr]^G
\]
onto the subspace $\widetilde{CSC}_G(E,P)$ of cyclic subgroup compatible elements.
\end{corollary}
\begin{proof}
Consider the following commutative diagram:
\begin{equation}\label{diag}
\begin{tikzcd}
    K_G(E, P)_\Q \arrow[r] \arrow[d] 
        & \displaystyle\bigoplus_{g \in G} K_{\cg g}(E|_{X^g}, P|_{E|_{X^g}})_\Q \arrow[d] \\
    \displaystyle\bigoplus_{g \in G} K_{\cg g}(E^g, P|_{E^g})_\Q \arrow[r] 
        & \displaystyle\bigoplus_{g \in G} \displaystyle\bigoplus_{i=0}^{|g|-1} K_{\cg{g^i}}(E^{g^i}, P|_{E^{g^{i}}})_\Q
\end{tikzcd}
\end{equation}
The upper horizontal map is as described in the theorem. The vertical maps are obtained by applying the decomposition theorem to $E$ and each $E|_{X^g}$ respectively. These are isomorphisms onto their images. 

The bottom right group is a direct sum over tuples $(g,i)$ with $g \in G$ and $0 \leq i < |g|$. Cyclic subgroup compatibility implies that the $(h,i)$- and $(k,j)$-summands must agree whenever $h^i = k^j$. 

The lower horizontal map is obtained by sending the element in the $g$-summand of the left-hand side to each $(h,i)$-summand on the right-hand side with $g = h^i$. This induces an isomorphism between the cyclic subgroup compatible elements of each side. Diagram \eqref{diag} therefore induces a diagram of isomorphisms:
\[
\begin{tikzcd}[column sep = 2pt]
    K_G(E,P)_\Q \arrow[rr, "\cong"] \arrow[rd, "\cong", swap]
        && \widetilde{CSC}_G(E,P)_\Q \arrow[ld, "\cong"] \\
    & CSC_G(E,P)_\Q &
\end{tikzcd}
\]
This completes the proof.
\end{proof}

\end{document}